\documentclass[12pt, reqno]{amsart}

\usepackage[utf8]{inputenc}
\usepackage[T1]{fontenc}
\usepackage{tikz,tikz-3dplot}
\usepackage{amsmath,amsfonts,amssymb,amsmath,latexsym,mathrsfs}
\usepackage{shuffle}
\usepackage[all]{xy}
\usepackage{stackengine}
 
\usepackage{enumerate}

\usepackage{hyperref}
 
\usepackage{tabu}
\usepackage{tikz-cd} 

\usepackage{comment}
\usepackage{url}
\usepackage{tikz}

\usepackage{xfrac}

\newtheorem{theorem}{Theorem}[section]

\newtheorem{proposition}[theorem]{Proposition}

\newcommand{\lt}{\triangleright}
\newcommand{\Conj}{\operatorname{Conj}}
\begin{document}

\title{On quandles with isomorphic enveloping groups}

\author{Mohamad \textsc{Maassarani} }
\maketitle
\begin{abstract}We give a method to construct two injective connected finite quandles of different cardinal with isomorphic enveloping groups. \end{abstract}
\section*{Introduction and main results}
For a finite dimensionl Lie algebra, the dimension of the Lie algebra is equal to the Gelfand-Kirillov dimension of its enveloping algebra (\cite{GK}). Hence, two finite dimensional Lie algebras with isomorphic enveloping algebras have the same dimension. This leads to the question wether two finite quandles with isomorphic enveloping groups have the same cardinal. For $Q$ a quandle the universal map $Q\to G(Q)$, where $G(Q)$ is the enveloping group, is not always injective and when it is not injective the image of this mapping defines a smaller quandle then $Q$ having enveloping group isomorphic to $G(Q)$ (\cite{GQ}). Hence, one can wonder wether two injective quandles ($Q\to G(Q)$ injective) with isomorphic enveloping groups has the same cardinal. We answer this question negatively. One can show that two finite connected affine quandles (Alexander quandles) with isomorphic enveloping groups have same cardinal and hence asks wether the question we raise before is true for connected quandles in general. We also answer this question negatively.\\\\
To answer the questions we raised, we consider finite centerless perfect groups $G$ with trivial Schur multipliers. A conjugacy class in $G$ has a quandle structure as a subquandle of the conjugacy quandle of $G$. We show that (as stated in \cite{Eis}) if the conjugacy class generates $G$ then the quandle it defines is injective, connected and with enveloping group isomorphic to $G\times \mathbb{Z}$. Hence, exibiting two generating conjugacy classes of $G$ with different sizes gives an exemple of injective connected quandles with isomorphic enveloping groups but with different sizes. An example of such conjugacy classes can be obtain from $G=PSL(2,8)$. The group $PSL(2,8)$ is simple and hence any non trivial conjugacy class generates the group and $PSL(2,8)$ has trivial Schur multiplier and non trivial conjugacy classes of different sizes. This is done in section 2. In section 1, we recall the definitions of quandles, their enveloping groups, related concepts and some properties.\\\\  
\section{Definitions and properties}
\emph{Quandles, quandle morphisms, and subquandles :} A quandle is a set $Q$ with a binary operation $\lt:Q\times Q\to Q$ such that
\begin{enumerate}
\item $x\lt x=x$ for all $x\in Q$;
\item for every $x\in Q$ the map $y\mapsto x\lt y$ for $y\in Q$ is bijective;
\item $x\lt(y\lt z)=(x\lt y)\lt(x\lt z)$ for all $x,y,z\in Q$.
\end{enumerate}
A morphism of quandles $\varphi:Q\to Q'$ is a map with $\varphi(x\lt y)=\varphi(x)\lt'\varphi(y)$. An isomorphism is a bijective morphism. A subset $S\subset Q$ is a subquandle if $x\lt y\in S$ for all $x,y\in S$ and $S$ is a quandle for the restricted operation.\\\\
\emph{Conjugacy quandles :} For a group $G$, the set $G$ with $x\lt y=xyx^{-1}$ is a quandle denoted $\Conj(G)$ and called the conjugacy quandle of $G$. Every union of conjugacy classes of $G$ is a subquandle of $\Conj(G)$. In particular a conjugacy class $C$ of $G$ is a quandle.\\\\
\emph{Inner automorphism group :} For $Q$ a quandle and $x \in Q$ the left translation $L_x$ by $x$ is the bijection given by $L_x(y)=x\lt y$ for $y\in Q$. The inner group $Inn(Q)$ is the subgroup of bijections of $Q$ generated by the left translations. The map $\theta_Q : Q\to Inn(Q),x\mapsto L_x$ is a quandle morphism with $Inn(Q)$ regarded as a conjugacy quandle. A quandle is called connected if it concists of one orbit under the action of $Inn(Q)$.\\\\
\emph{Enveloping groups :} The enveloping group of a quandle $Q$ is
$$G(Q)=\langle e_x,\ x\in Q\mid e_{x\lt y}=e_xe_ye_x^{-1}\ \ x,y\in Q\rangle .$$
For every group $H$ and quandle morphism $\varphi:Q\to \Conj(H)$ there is a unique group morphism $G(Q)\to H$ with $e_x\mapsto \varphi(x)$. A morphism $Q\to Q'$ induces a group morphism $G(Q)\to G(Q')$, so isomorphic quandles have isomorphic enveloping groups. The quandle $Q$ is called injective if $\varphi_Q:Q\to G(Q)$, $x\mapsto e_x$, is injective. We note that $\varphi_Q$ is a quandle morphism if we regared $G(Q)$ as a conjugacy quandle.\\
The abelianisation of $G(Q)$ is isomorphic to the free abelian group on the orbits of $Q$ with respect to $Inn(Q)$. The abelianisation map with this identification maps a generator $e_x$ to the orbit of $x$. Hence, if $Q$ is connected $G(Q)^{ab}\simeq \mathbb{Z}$.
As we have seen the mapping $\theta_Q : Q\to Inn(Q)$ is a quandle morphism for $Inn(Q)$ regarded as the conjugacy quandle. Hence we have a unique group morphism $\pi : G(Q) \to Inn(Q)$ such that $\pi(e_x)=L_x$ for $x\in Q$. If $Q$ is injective then $Ker(\pi)$ is exactly the center $Z$ of $G(Q)$ (\cite{JD}). In particular, $G(Q)/Z \simeq Inn(Q)$.\\\\
\emph{Affine quandles-Alexander quandles :} Let $A$ be an abelian group and $t$ be an automorhism of $A$. The Alexander quandle $Al(A,t)$ is the set $A$ equipped with the operation $x\lt y=ty+(1-t)x$. An affine quandle is an Alexander quandle. For $A$ finite the quandle $Al(A,t)$ is connected if and only if $1-t$ is an automorphism of $A$, in wich case the inner automorphism group is isomorphic to the semidirect product $A\rtimes <t>$. The Derived subgroup of the inner automorphism group in this case is $A$. Alexander quandles are injective. Hence, the derived group of $G(Al(A,t))/Z\simeq Inn(Al(A,t))$ (with $Z$ the center of the group) is $A$. It follows that two finite connected Alexander quandles with isomorphic enveloping groups has the same cardinal. 
\section{Enveloping groups of generating conjugacy classes of some groups}

\begin{proposition}
Let $G$ be a finite centerless group and $C$ be a conjugacy class $($and hence a subquandle$)$ generating $G$. 
\item[1)] $Inn(C)$  is isomorphic to $G$.
\item[2)] $C$ is a connected quandle.
\item[3)] $C$ is an injective quandle.
\item[4)] We have an exact sequence $1\to Z \to G(C) \to G \to 1$ where $Z$ is the center of $G(C)$.
\item[5)] If the Schur multiplier of $G$ is trivial then $Z\cap [G(C),G(C)]=1$.
\end{proposition}
 \begin{proof}
Since $C$ generates $G$ the conjugacy by an element of $G$ restricted to $C$ is the product of left translations of $C$. This gives a group morphism $G\to Inn(C)$. This morphism is injective since $C$ generates $G$. Indeed, if the conjugacy by an element of $G$ fixes $C$ pointwise, hence it fixes $G$ pointwise so it is central and hence equal to $1$, since $G$ is centerless. This proves $1)$. $C$ is clearly connected since any two element of $C$ can be related by the conjugacy by an element of $G$. We have proved $2)$. We prove $3)$. The inclusion $C\to G$ is a quandle inclusion into $\Conj(G)$. Hence, there exists a group morphism $G(C) \to G$ such that $e_c\mapsto c$ for $c\in C$. This proves that $\varphi_C:C\to G(C),c\mapsto e_c$ is injective and hence $C$ is injective by definition. $4)$ follows from $1)$ and $3)$ since $G(C)/Z\simeq Inn(C)$ if $C$ is injective as we have seen in the previous section. $5)$ follows from proposition 3.2 and corollary 3.9 of \cite{Ir}.
\end{proof}
\begin{proposition}
Let $G$ be a finite centerless perfect group with trivial Schur multiplier and $C$ be a conjugacy class generating $G$. The enveloping group $G(C)$ of $C$ is isomorphic to $G\times \mathbb{Z}$.
\end{proposition}
\begin{proof}
By the previous proposition $C$ is connected. Hence, $G(C)^{ab}\simeq \mathbb{Z}$ (seen in the previous section). Let $Ab : G(C)\to \mathbb{Z}$ be the morphism corresponding to the abelianisation morphism. Denote by $f:G(C)\to G$ the morphism of the exact sequence of the previous proposition. Consider $f\times Ab: G(C)\to G\times \mathbb{Z}$. The morphism $f\times Ab$ is injective by $5)$ of the previous proposition. Now $G\times\{0\}$ lies in the image of $G(Q)$. Indeed, the image of $[G(C),G(C)]$ lies in $G\times \{0\}$ and maps onto $[G,G]\times \{0\}=G\times \{0\}$ since $f$ is surjective. The fact that $G\times\{0\}$ lies in the image and $Ab$ is surjective proves that $f\times Ab$ is surjective. We have proven that $f\times Ab$ is injective and surjective. This proves the proposition.
\end{proof}
It follows from the last proposition and the one before that if $C$ and $C'$ are conjugacy classes of different size generating a group $G$ as in the last proposition, then $C$ and $C'$ give two injective connected quandles of different sizes with isomorphic enveloping groups (groups isomorphic to $G\times \mathbb{Z}$). This is the construction we explained in the introduction. For a simple group a non trivial conjugacy class generates the group (see the proof below) and hence a simple group with trivial Schur multiplier and non trivial conjugacy classes with different sizes provides an example. The group $PSL(2,8)$ satisfies these conditions :
\begin{proposition}
\item[1)] $PSL(2,8)$ is simple and hence centerless and perfect.
\item[2)] $PSL(2,8)$ has trivial Schur multiplier.
\item[3)] A non trivial conjugacy class of $PSL(2,8)$ generates it.
\item[4)] $PSL(2,8)$ has $9$ conjugacy classes with sizes given by the list  $[ 1, 72, 72, 72, 63, 56, 56, 56, 56 ]$.
\item[5)] The quandles defined by a non trivial conjugacy classes of $PSL(2,8)$ are finite connected injective quandles and have enveloping groups isomorphic to $G\times \mathbb{Z}$.
\end{proposition}
\begin{proof}
Items $1)$ to $4)$ seem to be known in the literature. We give a verification using GAP. The group $PSL(2,8)$ correspond to the group SmallGroup(504,156) of the GAP library. As one can check with GAP the code  G:=SmallGroup(504,156);StructureDescription(G); returns "PSL(2,8)". The code IsSimple(G); returns true hence $PSL(2,8)$ is simple as in $1)$. The code AbelianInvariantsMultiplier(G); gives a list of the abelian invariants of the Schur multiplier of $G$, it returns an empty list and hence $PSL(2,8)$ has trivial Schur multiplier as in $2)$. A non trivial conjugacy class of $PSL(2,8)$ generates a non trivial normal subgroup. Since the group is simple the conjugacy class generates the group. This proves $3)$. The code C:=ConjugacyClasses(G); List(C,Size); gives a list of sizes of the conjugacy classes of the group. The code returns the list in $4)$. $5)$ follows from $1)$ to $3)$ of this proposition and the last two propositions.
\end{proof}
Points $4)$ and $5)$ of the last proposition proves that we have connected injective quandles of size $72,63$ and $56$ with isomorphic enveloping groups and answer the question raised in the introduction.

 \end{document}